\documentclass[a4paper,10pt]{article}
\usepackage{amsmath,amssymb,fullpage,graphicx,nameref,amsthm, tikz,color,microtype,datetime,hyperref}
\title{Colouring of plane graphs with unique maximal colours on faces}
\author{Alex Wendland\thanks{Warwick Institute of Mathematics, University of Warwick, Coventry, United Kingdom, CV4 7AL. E-mail: \texttt{a.p.wendland@warwick.ac.uk}. This research was done during the author's visit to Charles University in Prague and University of West Bohemia in Pilsen which was supported by Undergraduate Research Support Scheme of the University of Warwick and the grant GA14-19503S (Graph coloring and structure) of the Czech Science Foundation.}}
\date{}
\usetikzlibrary{arrows,shapes}
\newcommand{\tab}{\hspace*{0.2 in}}
\newtheorem{mydef}{Definition}[section]
\newtheorem{thm}[mydef]{Theorem}
\newtheorem{lem}[mydef]{Lemma}
\newtheorem{prop}[mydef]{Proposition}

\newtheorem{conj}[mydef]{Conjecture}
\theoremstyle{definition}

\begin{document}

\maketitle

\begin{abstract}
The Four Colour Theorem asserts that the vertices of every plane graph can be properly coloured with four colours. Fabrici and G\"oring conjectured the following stronger statement to also hold: the vertices of every plane graph can be properly coloured with the numbers $1,\ldots,4$ in such a way that every face contains a unique vertex coloured with the maximal colour appearing on that face. They proved that every plane graph has such a colouring with the numbers $1,\ldots,6$. We prove that every plane graph has such a colouring with the numbers $1,\ldots,5$ and we also prove the list variant of the statement for lists of sizes seven.
\end{abstract}

\section{Introduction}

\tab A lot of research in graph theory was sparked by the problem of four colours posed by Francis Guthrie in 1852. It took more than 125 years until the problem was resolved by Appel and Haken \cite{KAHW} and the conjectured statement became known as the Four Colour Theorem. A refined proof of the Four Colour Theorem was given by Robertson, Seymour, Sanders and Thomas \cite{RSST}. Our work is motivated by a conjecture of Fabrici and G\"oring \cite{IFFG} which, if true, would strengthen the Four Colour Theorem.
\begin{conj}
\label{conj}
(Fabrici and G\"oring \cite[Conjecture 9]{IFFG}) Every plane graph has a proper colouring using the number $1$, $2$, $3$ and $4$ such that every face contains a unique vertex coloured with the maximal colour appearing on that face.
\end{conj}
We will refer to a colouring of this kind as to a {\em capital} colouring, i.e., a capital colouring is a proper vertex colouring using integers such that every face contains a unique vertex coloured with the maximal colour appearing on that face. The name comes from the fact that every face (region) has a unique vertex (capital) with the maximal colour. The capital chromatic number $\chi_C(G)$ of a graph $G$ is the smallest $k$ such that there exists a capital colouring using $1, \ldots, k$. Here we would like to note that we state our results using plane graphs, graphs with an embedding into the plane, instead of planar graphs, graphs such that there exists an embedding into the plane so does not have a fixed embedding. A face of a plane graph is the maximal connected part of the plane with the drawing of the graph removed and we often associate a face with the vertices and faces that bound it.\\

Note that Conjecture 1.1 holds for triangulations since any proper colouring of a triangulation has the required properties. Fabrici and G\"oring~\cite{IFFG} proved that every plane graph has a capital colouring using colours $1,\ldots,6$. We prove a stronger result that every plane graph has a capital colouring using colours $1,\ldots,5$.

\begin{thm}
If $G$ is a plane graph then $\chi_{C}(G) \leq 5$. \label{my1}
\end{thm}

In addition, we consider the list version of capital colourings and we show that if each vertex of a plane graph is assigned a list of seven integers, then there exists a capital colouring assigning each vertex a colour from its list. Throughout this paper, a plane graph is a loopless graph embedded in the plane that may contain parallel edges which may (but need not) form 2-faces.

\section{Unique maximum 5-colouring}

We start by recalling an auxiliary Lemma \ref{FG2} from \cite{IFFG}. 

\begin{lem}
(Fabrici and G\"{o}ring \cite[Lemma 6]{IFFG}) Let $G$ be a plane graph with no parallel edges, let $xy \in E(G)$ be an edge of $G$ incident with the outer face, and let $c \in \{\mbox{black}, \mbox{blue}\}$. There is a non-proper 3-vertex-colouring of $G$ with colours red, blue and black such that
\begin{enumerate}
\item
vertex $x$ has colour $c$,
\item
vertex $y$ is black,
\item
each edge is incident with at most one blue vertex,
\item
no vertex incident with the outer face is red,
\item
each inner face is incident with at most one red vertex, and
\item
each inner face that is not incident with a red vertex is incident with exactly one blue vertex.
\end{enumerate}
\label{FG2} 
\end{lem}

The proof of Theorem \ref{my1} uses a stronger version of Lemma \ref{FG2}. The version differs by adding the condition that all triangles contain at least one blue or red vertex. Throughout the following proof we use the terminology separating cycles, which is a cycle such that when removed it disconnects the graph.

\begin{lem}
Let $G$ be a plane graph without 2-faces, let $xy \in E(G)$ be an edge of $G$ incident with the outer face, and let $c \in \{\mbox{black}, \mbox{blue}\}$. There is a non-proper 3-vertex-colouring of $G$ with colours red, blue and black such that
\begin{enumerate}
\item
vertex $x$ has colour $c$,
\item
vertex $y$ is black,
\item
each edge is incident with at most one blue vertex,
\item
no vertex incident with the outer face is red,
\item
each inner face is incident with at most one red vertex,
\item
each inner face that is not incident with a red vertex is incident with exactly one blue vertex, and
\item
each triangle contains at least one vertex that is not black.
\end{enumerate}
\label{my2} 
\end{lem}
\begin{proof}
We proceed by induction on the number of vertices. Let $xy \in E(G)$ be an edge of $G$ incident with the outer face and $c \in \{\mbox{black}, \mbox{blue}\}$. If $G$ has no separating cycles of length two or three, then Lemma~\ref{FG2} yields the statement unless the outer face is 3-face (note that the outer face is not a 2-cycle since $G$ has no 2-faces). If $c =$ blue and the outer face is a 3-face, Lemma~\ref{FG2} also yields the statement. If $c =$ black, switch $x$ to the vertex of the outer 3-face different from $x$ and $y$, let $c =$ blue and apply Lemma~\ref{FG2}. Note that the vertices that were originally $x$ and $y$ must be black since $G$ has no edge with two blue end vertices.\\

Assume there are separating cycles of length two and consider an inner most separating cycle of length two $C$. Use $u_1$ and $u_2$ for the vertices of this cycle. Let $G_1$ be the graph contained strictly outside $C$ and $G_2$ be the graph contained strictly inside $C$, therefore the vertex sets of $G_1$, $G_2$ and $C$ partition the vertex set of $G$. We call the graph induced on the vertices of $G_1$ and $C$ $G_1 \cup \{u_1, u_2\}$ whereas we call $G_1 \cup \{u_1, u_2\}$ without one of the edges contained in $C$ the graph $G_1 + u_1u_2$. See Figure 1 for set up.

\begin{center}
\begin{tikzpicture}[scale = 1]
\draw (0,0) circle (0.25 cm);
\draw (0,0) circle (1.5 cm);
\node at (1.9,0) {$G_1$};
\node at (0.6,0) {$C$};
\node at (0,0) {$G_2$};
\node (G-11) at (0,1) [circle,fill=black] [label=left:$u_1$] {};
\node (G-12) at (0,-1) [circle,fill=black] [label=left:$u_2$] {};
\draw[-][black] (G-11) edge [bend left] (G-12);
\draw[-][black] (G-12) edge [bend left] (G-11);
\end{tikzpicture}\\
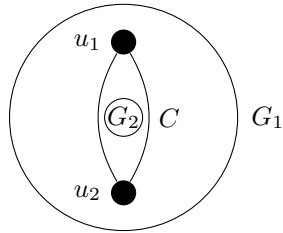
Figure 1: Cycle $C$.
\end{center}

$G_2$ is not empty as $C$ is a separating cycle. Next apply the induction assumption on the graph $G_1 + u_1u_2$ (i.e. there is only one of the edges of $C$ present). Then given $xy$ and $c$, induction hypothesis guarantees a 3-colouring of $G_1 + u_1u_2$ with the properties as desired in the statement. As $u_1u_2$ is an edge, there are two possibilities for the colours of the vertices $u_1$ and $u_2$: one vertex is red and the other is coloured with $c'$ and one vertex is black and the other coloured with $c'$, where $c' \in \{\mbox{blue}, \mbox{black}\}$. See figure 2 for these cases.\\

\begin{center}
\begin{tikzpicture}[scale=1]
\draw (0,0) circle (0.27 cm);
\draw (0,0) circle (1.5 cm);
\node at (1.9,0) {$G_1$};
\node at (0.6,0) {$C$};
\node at (0,0) {$G_2$};
\node (G-11) at (0,1) [diamond,fill=black] [label=left:$u_1$] {};
\node (G-12) at (0,-1) [rectangle,fill=black] [label=left:$u_2$] {};
\draw[-][black] (G-11) edge [bend left] (G-12);
\draw[-][black] (G-12) edge [bend left] (G-11);
\node at (-5,1) [circle, fill=black] [label=right:black] {};
\node at (-5,0) [rectangle, fill=black] [label=right:c'] {};
\node at (-5,-1) [diamond,fill=black] [label=right:red] {};
\end{tikzpicture}
\hspace*{0.5 in}
\begin{tikzpicture}[scale=1]
\draw (0,0) circle (0.27 cm);
\draw (0,0) circle (1.5 cm);
\node at (1.9,0) {$G_1$};
\node at (0.6,0) {$C$};
\node at (0,0) {$G_2$};
\node (G-11) at (0,1) [circle,fill=black] [label=left:$u_1$] {};
\node (G-12) at (0,-1) [rectangle,fill=black] [label=left:$u_2$] {};
\draw[-][black] (G-11) edge [bend left] (G-12);
\draw[-][black] (G-12) edge [bend left] (G-11);
\end{tikzpicture}\\
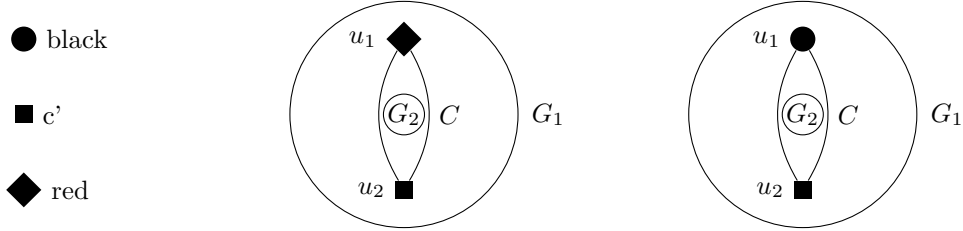
Figure 2: Two possibilities for cycle $C$ after colouring $G_1 \cup \{u_1, u_2\}$ with inductive assumption.
\end{center}

\paragraph{One vertex red and one vertex coloured with $c'$.} Assume that $u_1$ is red. Consider the induced graph on $V(G_2) \cup \{u_2\}$. Then if $u_2$ is joined to a vertex in $G_2$ take an edge $u_2 v$ on the outer face. If $u_2$ is not joined to a vertex in $G_2$ take a vertex $v$ on the outer face of $G_2$, and add edge $u_2 v$. Then apply the inductive assumption to the constructed graph with $u_2$ as $x$ with colour $c'$ and $v$ as $y$ coloured black. The constructed 3-colouring matches up with the one of $G_1 + u_1u_2$ and gives the desired 3-colouring.

\paragraph{One vertex black and one vertex coloured with $c'$.} Say that $u_1$ is black. Since $C$ does not bound a face, the graph $G_2+u_1u_2$ contains another vertex on its outer face. Let $v$ be this vertex. Apply the inductive assumption to the graph $(G_2+u_1u_2)\backslash\{v\}$ with $u_1$ as $y$ coloured black and $u_2$ as $x$ coloured $c'$. Then colour $v$ red to get a 3-colouring on $G_2 \cup \{u_1, u_2\}$. The constructed 3-colouring matches up with the one on $G_1 + u_1u_2$ and gives the desired 3-colouring.\\

Assume there are no separating cycles of length two but $G$ has separating cycles of length three. Let $T$ be an inner most triangle. Use $t_1$, $t_2$ and $t_3$ for the vertices of this triangle. Let $G_1$ be the graph strictly contained outside $T$ and $G_2$ be the graph strictly contained inside $T$. See figure 3 for set up.

\begin{center}
\begin{tikzpicture}[scale = 0.7]
\draw (1.5,1) circle (0.5cm);
\draw (1.5,1) circle (2.7cm);
\node at (-0.8,3.2) {$G_1$};
\node at (2.4,1.8) {$T$};
\node at (1.5,1) {$G_2$};
\node (G-11) at (0,0) [circle,fill=black] [label=left:$t_3$] {};
\node (G-12) at (3,0) [circle,fill=black] [label=right:$t_2$] {};
\node (G-13) at (1.5,2.6) [circle,fill=black] [label=above:$t_1$] {};
\draw[-][black] (G-11) edge (G-12);
\draw[-][black] (G-12) edge (G-13);
\draw[-][black] (G-13) edge (G-11);
\end{tikzpicture}\\
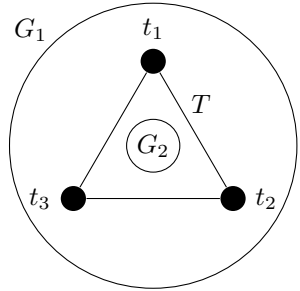
Figure 3: Triangle $T$.
\end{center}

$G_2$ is not empty as $T$ is a separating triangle. Apply the induction assumption on the graph $G_1 \cup T$ with $xy$ and $c$ to get a 3-colouring of $G_1 \cup T$ with the properties as desired in the statement. As $T$ bounds an inner face in $G_1 \cup T$, some of its vertices must be coloured with blue or red and there are three possibilities: one vertex is red, blue and black, one vertex is red and two are black, or one vertex is blue and two are black. We now consider these three cases, demonstrated in figure 4.\\

\begin{center}
\begin{tikzpicture}[scale=0.7]
\draw (1.5,1) circle (0.5cm);
\draw (1.5,1) circle (2.7cm);
\node at (-0.8,3.2) {$G_1$};
\node at (2.4,1.8) {$T$};
\node at (1.5,1) {$G_2$};
\node at (-4,2) [circle, fill=black] [label=right:black] {};
\node at (-4,1) [rectangle, fill=black] [label=right:blue] {};
\node at (-4,0) [diamond,fill=black] [label=right:red] {};
\node (G-11) at (0,0) [circle, fill=black] [label=left:$t_3$] {};
\node (G-12) at (3,0) [rectangle, fill=black] [label=right:$t_2$] {};
\node (G-13) at (1.5,2.6) [diamond,fill=black] [label=above:$t_1$] {};
\draw[-][black] (G-11) edge (G-12);
\draw[-][black] (G-12) edge (G-13);
\draw[-][black] (G-13) edge (G-11);
\end{tikzpicture} \hspace*{0.1 in}
\begin{tikzpicture}[scale=0.7]
\draw (1.5,1) circle (0.5cm);
\draw (1.5,1) circle (2.7cm);
\node at (-0.8,3.2) {$G_1$};
\node at (2.4,1.8) {$T$};
\node at (1.5,1) {$G_2$};
\node (G-11) at (0,0) [circle, fill=black] [label=left:$t_3$] {};
\node (G-12) at (3,0) [circle, fill=black] [label=right:$t_2$] {};
\node (G-13) at (1.5,2.6) [diamond, fill=black] [label=above:$t_1$] {};
\draw[-][black] (G-11) edge (G-12);
\draw[-][black] (G-12) edge (G-13);
\draw[-][black] (G-13) edge (G-11);
\end{tikzpicture} \hspace*{0.1 in}
\begin{tikzpicture}[scale=0.7]
\draw (1.5,1) circle (0.5cm);
\draw (1.5,1) circle (2.7cm);
\node at (-0.8,3.2) {$G_1$};
\node at (2.4,1.8) {$T$};
\node at (1.5,1) {$G_2$};
\node (G-11) at (0,0) [circle, fill=black] [label=left:$t_3$] {};
\node (G-12) at (3,0) [circle, fill=black] [label=right:$t_2$] {};
\node (G-13) at (1.5,2.6) [rectangle, fill=black] [label=above:$t_1$] {};
\draw[-][black] (G-11) edge (G-12);
\draw[-][black] (G-12) edge (G-13);
\draw[-][black] (G-13) edge (G-11);
\end{tikzpicture}\\
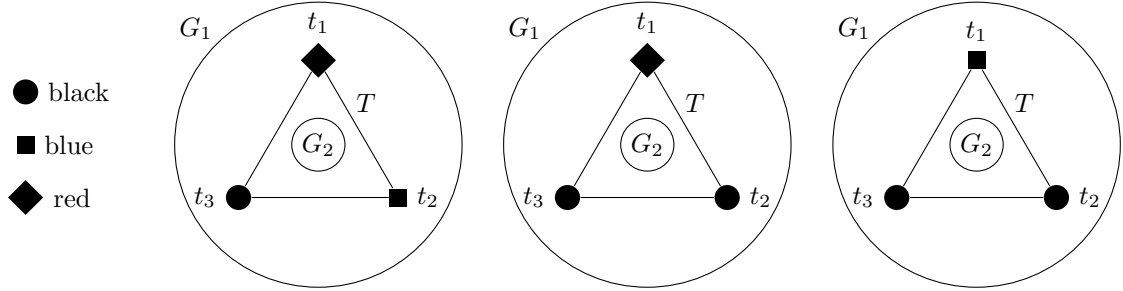
Figure 4: Three possibilities for triangle $T$ after colouring $G_1 \cup T$ with inductive assumption.
\end{center}

\paragraph{One vertex of each colour.} Without loss of generality assume $t_1$ is red, $t_2$ is blue and $t_3$ is black. Apply the inductive assumption on the graph induced by $V(G_2) \cup \{t_2, t_3\}$ with $t_2$ as $x$ to be coloured blue, which is the colour $c$, and $t_3$ as $y$ to be black. The 3-colourings on $G_1 \cup T$ and the graph induced by $V(G_2) \cup \{t_2, t_3\}$ match up and give a colouring of $G$ having the desired properties.

\paragraph{Two black vertices and one red.} Assume $t_1$ is red. The inductive assumption is applied on the graph induced by $V(G_2) \cup \{t_2, t_3\}$ again with $t_2t_3$ being $xy$ to be both coloured black. The 3-colourings on $G_1 \cup T$ and the graph induced by $V(G_2) \cup \{t_2, t_3\}$ match up on $t_2t_3$ and give us a 3-colouring as described in the statement of Lemma \ref{my2}.

\paragraph{Two black vertices and one blue.} Let $t_1$ be blue. Apply the inductive assumption on the graph $G_2 \cup T$ with $t_1$ as $x$ to be coloured blue and $t_2$ as $y$ to black. The 3-colourings on $G_1 \cup T$ and $G_2 \cup T$ match up ($t_3$ cannot be coloured blue as it is connected to $t_1$ therefore as it is on the outer face it must be black) and give us the required 3-colouring.
\end{proof}

Lemma \ref{my2} in conjunction with Gr\"{o}tzsch theorem, yields a proof of Theorem \ref{my1}. We recall the statement of Gr\"{o}tzsch theorem and complete the proof of Theorem \ref{my1}.

\begin{thm}
(Gr\"{o}tzsch \cite{HG59}) Every triangle-free planar graph $G$ has a proper 3-colouring.
\end{thm}

\begin{proof}[Proof of Theorem \ref{my1}]
The goal of this proof is to find a non-proper 3-colouring of $G$ that has either a single red vertex or no red vertex and a single blue vertex on each face but won't abide by all the rules of Lemma~\ref{my2}. Then to convert this non-proper 3-colouring to a Capital colouring of $G$. If $G$ has 2-faces replace these with a single edge, then a capital colouring of the altered graph is a capital colouring of the original graph. So without loss of generality assume $G$ has no 2-faces. Choose a vertex $v \in V(G)$ on the outer face and then apply Lemma \ref{my2} to the graph $G \backslash v$, picking any edge $xy$ on the outer face with any colour $c \in \{\mbox{black}, \mbox{blue}\}$. Then let $v$ be coloured red to get a 3-colouring. Note that each face has either exactly one red vertex (such as any face containing $v$), or no red vertex and exactly one blue vertex. Moreover, every triangle contains at least one red or blue vertex. Let $H$ be the subgraph of $G$ induced by the black vertices. As $H$ is triangle-free, by Gr\"{o}tzsch theorem, there exists a proper 3-colouring of it using $\{1,2,3\}$. Then assign blue vertices the colour $4$ and red vertices the colour $5$. The constructed 5-colouring is proper and has a unique maximal colour on each face from the construction.
\end{proof}

\section{List Colouring}

\tab In this section we will present an upper bound for the capital list colouring of a plane graph G.

\begin{mydef}
A list assignment is a function $L: V(G) \rightarrow \mathcal{P}(\mathbb{N})$. A graph $G$ has a capital $L$-colouring if it has a capital colouring $c : V(G) \rightarrow \mathbb{N}$ such that $c(v) \in L(v)$ for every $v \in V(G)$. We say that a graph $G$ is capital $k$-choosable if there is a capital $L$-colouring for all list assignments with $\vert L(v) \vert \geq k$ for all $v \in V(G)$. The minimum $k$ such that $G$ is capital $k$-choosable is denoted by $\chi_{C}^l(G)$.
\end{mydef}

We will prove an upper bound of seven on $\chi_{C}^l(G)$ for any plane graph $G$.

\begin{thm}
If $G$ is a plane graph then $\chi_{C}^l(G) \leq 7$. \label{my3}
\end{thm}

The proof of Theorem \ref{my3} shall use a discharging argument. We assume Theorem \ref{my3} is false and take $G$ to be an extremal counter-example with a list assignment $L$. We say $G$ is extremal if we partially order counter examples by the following criteria and pick the $G$ to be a minimal graph in this ordering, so it has the minimum number of vertices, the minimum number of 2-faces and the maximal number of edges (in this order). Lets first discus why we can pick such a graph $G$.\\

\begin{lem}
Assuming Theorem \ref{my3} is false an extremal counter example $G$ exists.
\end{lem}
\begin{proof}
As Theorem \ref{my3} is false a counter example exists, therefore we can partially order the counter examples with respect to the criteria above. As the number of vertices and the number of 2-faces are countable we can find a set of extremal counter examples with respect to these criteria. Then with a set number of vertices and 2-faces the number of edges is bounded, therefore we can pick a $G$ with the maximum number of edges.
\end{proof}

Before examining properties of a extremal counter-example, we introduce some notation. A vertex of degree $d$ is called a $d$-vertex and a $\geq d$-vertex is a vertex of degree at least $d$. A $d$-face is a face incident with exactly $d$ edges and $\geq d$-face is a face incident with at least $d$ edges. If $f$ is a face, then we write $c(f)$ to be the maximal colour of $f$ under a colouring $c$.

\subsection{Reducible Configurations}

\tab In this subsection, we will explore properties of a extremal counter-example $G$ with list colouring $L$.

\begin{lem}
Let $G$ be a extremal counter-example. \label{my4}
\begin{enumerate}
 \item
 $G$ is 2-connected: in other words $G$ is connected and no vertex can be removed to disconnected $G$.
 \item
 For all vertices $v \in V(G)$, if $v$ is adjacent to $k$ vertices and $l$ faces of size at least four then $k + l \geq 7$.
 \item 
 Each vertex of $G$ is a $\geq 4$-vertex.
 \item
 Each face of $G$ is a $\geq 3$-face. 
 \item
 No two 3-faces share a 4-vertex.
\end{enumerate}
\end{lem}
\begin{proof}
\begin{enumerate}
 \item
If $G$ is disconnected, then it has different components $G_1$ and $G_2$. Pick a vertex on the face shared by $G_1$ and $G_2$ for both $G_1$ and $G_2$ and then add an edge between the two. By the extremity of $G$, specifically the maximality in terms of edges as we are not increasing the number of vertices or 2-faces, we can find an $L$-colouring. This colouring is also a colouring of the original graph.\\

Suppose $G$ has a cut vertex $v$ with a face $f$ on two sides, let $v_1vv_2$ and $v_3vv_4$ be walks on the boundary of face $f$, see Figure 5. Consider the graph $H$ which is $G$ with an additional edge $v_1v_2$. By extremity of $G$, $H$ has an $L$-colouring as $H$ has no vertices or 2-faces as any 2-face would be one of $G$. Then this is a capital $L$-colouring on $G$.

\begin{center}
\begin{tikzpicture}[scale = 0.8]
\node at (0,1.5) {$f$};
\node at (0,1.5) {$f$};
\node (G-1) at (0,0) [circle,fill=black] [label=above:$v$] {};
\node (G-2) at (-1,1) [circle,fill=black] [label=above:$v_1$] {};
\node (G-3) at (-1,-1) [circle,fill=black] [label=below:$v_3$] {};
\node (G-4) at (1,1) [circle,fill=black] [label=above:$v_2$] {};
\node (G-5) at (1,-1) [circle,fill=black] [label=below:$v_4$] {};
\node (G-21) at (2,-3) {};
\node (G-22) at (3,-1) {};
\node (G-23) at (3,-1) {};
\node (G-24) at (2,3) {};
\draw (-1,-1) --(-2,-1.5) -- (-3,-1) -- (-3,1) -- (-2, 1.5) -- (-1,1);
\draw (1,-1) --(2,-1.5) -- (3,-1) -- (3,1) -- (2, 1.5) -- (1,1);
\draw[-][black] (G-1) edge (G-2);
\draw[-][black] (G-1) edge (G-3);
\draw[-][black] (G-1) edge (G-3);
\draw[-][black] (G-1) edge (G-4);
\draw[-][black] (G-1) edge (G-5);
\end{tikzpicture}
\hspace*{0.5 in}
\begin{tikzpicture}[scale = 0.8]
\node at (0,1.5) {$f$};
\node at (0,1.5) {$f$};
\node (G-1) at (0,0) [circle,fill=black] [label=above:$v$] {};
\node (G-2) at (-1,1) [circle,fill=black] [label=above:$v_1$] {};
\node (G-3) at (-1,-1) [circle,fill=black] [label=below:$v_3$] {};
\node (G-4) at (1,1) [circle,fill=black] [label=above:$v_2$] {};
\node (G-5) at (1,-1) [circle,fill=black] [label=below:$v_4$] {};
\node (G-21) at (2,-3) {};
\node (G-22) at (3,-1) {};
\node (G-23) at (3,-1) {};
\node (G-24) at (2,3) {};
\draw (-1,-1) --(-2,-1.5) -- (-3,-1) -- (-3,1) -- (-2, 1.5) -- (-1,1);
\draw (1,-1) --(2,-1.5) -- (3,-1) -- (3,1) -- (2, 1.5) -- (1,1);
\draw[-][black] (G-1) edge (G-2);
\draw[-][black] (G-1) edge (G-3);
\draw[-][black] (G-1) edge (G-3);
\draw[-][black] (G-1) edge (G-4);
\draw[-][black] (G-1) edge (G-5);
\draw[-][black] (G-2) edge (G-4);
\end{tikzpicture}\\
Figure 5: Configuration of Proposition \ref{my4} and reduction.
\end{center}

\item
Suppose $G$ has a vertex $v \in V(G)$ such that $k + l \leq 6$. Let $v_1, \ldots, v_k$ be the neighbours of $v$ in the cyclic order. Remove $v$ and add edges $v_1v_2$, $v_2v_3$, $\ldots$ $v_{k-1}v_k$ and $v_kv_1$. By extremity of $G$ we can colour the remaining graph from the lists $L$, as the graph remaining has less vertices. Then assign a colour to $v$ from its list that is not assigned to its neighbours and that is not the maximal colour on any of the incident faces. Since there are at most $k+l$ such colours, there is a colour in $L(v)$ that can be assigned to $v$. On the faces containing $v$, the maximal colour of the face is either on the vertex that had the maximal colour in the modified graph or on $v$. Therefore $G$ has a capital $L$-colouring.
\item
This follows from Part 2 as $l \leq k$ for any vertex.
\item
If $f$ is a 2-face remove one of the edges. By the extremity of $G$, we can find a capital $L$-colouring. This colouring is still a capital colouring of the original graph.
\item
Follows directly from Part 2.
\end{enumerate}
\end{proof}

We now look at a 4-face sharing an edge with a 3-face.

\begin{prop}
 In a extremal counter-example $G$, no 3-face and 4-face can share an edge joining two 4-vertices. \label{my5}
 \end{prop}
 \begin{proof}
 Assume Proposition \ref{my5} is false and there exists a extremal counter-example $G$ with such a configuration. Let the 3-face be $v_1v_2v_3$ and the 4-face $v_1v_2v_5v_4$ with $v_1$ and $v_2$ being 4-vertices. Let $v_7$ be the remaining vertex connected to $v_1$, $f_1$ the face bounded partially by $v_3v_1v_7$ and $f_2$ the face partially bounded by $v_7v_1v_4$. Let $v_6$ be the remaining vertex connected to $v_2$, $f_3$ the face partially bounded by $v_5v_2v_6$ and $f_4$ the face partially bounded by $v_6v_2v_3$. See Figure 6.
 
 \begin{center}
\begin{tikzpicture}[scale = 0.8]
\node at (-1,0) {$f_1$};
\node at (0.5,1.5) {$f_2$};
\node at (2.5,-2) {$f_3$};
\node at (0,-3) {$f_4$};
\node (G-1) at (0,0) [circle,fill=black] [label=right:$v_1$] {};
\node (G-2) at (-1,-2) [circle,fill=black] [label=left:$v_3$] {};
\node (G-3) at (1,-2) [circle,fill=black] [label=right:$v_2$] {};
\node (G-4) at (1,1) [circle,fill=black] [label=right:$v_4$] {};
\node (G-5) at (2,-1) [circle,fill=black] [label=right:$v_5$] {};
\node (G-6) at (2,-3) [circle,fill=black] [label=right:$v_6$] {};
\node (G-7) at (0,1) [circle,fill=black] [label=left:$v_7$] {};
\draw[-][black] (G-1) edge (G-2);
\draw[-][black] (G-1) edge (G-3);
\draw[-][black] (G-2) edge (G-3);
\draw[-][black] (G-1) edge (G-4);
\draw[-][black] (G-1) edge (G-7);
\draw[-][black] (G-3) edge (G-5);
\draw[-][black] (G-3) edge (G-6);
\draw[-][black] (G-5) edge (G-4);
\end{tikzpicture}
\hspace{0.5 in}
\begin{tikzpicture}[scale = 0.8]
\node at (-1,0) {$f_1'$};
\node at (0.5,1.5) {$f_2'$};
\node at (2.5,-2) {$f_3'$};
\node at (0,-3) {$f_4'$};
\node (G-2) at (-1,-2) [circle,fill=black] [label=left:$v_3$] {};
\node (G-4) at (1,1) [circle,fill=black] [label=right:$v_4$] {};
\node (G-5) at (2,-1) [circle,fill=black] [label=right:$v_5$] {};
\node (G-6) at (2,-3) [circle,fill=black] [label=right:$v_6$] {};
\node (G-7) at (0,1) [circle,fill=black] [label=left:$v_7$] {};
\draw[-][black] (G-7) edge (G-2);
\draw[-][black] (G-2) edge (G-6);
\draw[-][black] (G-6) edge (G-5);
\draw[-][black] (G-4) edge (G-7);
\draw[-][black] (G-5) edge (G-4);
\end{tikzpicture}\\
Figure 6: Configuration of Proposition \ref{my5} and reduction.
\end{center}

Let $H$ be the graph $G \backslash \{v_1, v_2\}$ but with the additional edges $v_4v_7$, $v_7v_3$, $v_3v_6$ and $v_6v_5$. Let $f_1'$ be the new face partially bounded by $v_3v_7$, $f_2'$ by $v_7v_4$, $f_3'$ by $v_5v_6$ and $f_4'$ by $v_6v_3$. Then by the extremity of $G$, $H$ has an $L$-colouring $c$. As $v_4$ and $v_5$ are adjacent by symmetry we can assume $c(v_4) > c(v_5)$. Colour $v_2$ from $L(v_2)$ by a colour different from $c(v_3), c(v_4), c(v_5), c(v_6), c(f_3')$ and $c(f_4')$; call this colour $c(v_2)$. Then colour $v_1$ from $L(v_1)$ by a colour different from $c(v_2), c(v_3), c(v_4), c(v_7), c(f_1')$ and $c(f_2')$. The resulting colouring is a capital $L$-colouring. Indeed, the maximal colour on the face $v_1v_2v_5v_4$ is the colour of either $v_1$, $v_2$ or $v_4$. Therefore a capital $L$-colouring of $G$ exists contradicting that $G$ is a extremal counter-example.
\end{proof}

The proofs of Proposition \ref{my6} and \ref{my7} use the same idea as in the proof of Proposition \ref{my5}.

\begin{prop}
 In a extremal counter-example $G$, no 3-face and 4-face can share an edge joining a 4-vertex and a 5-vertex incident to three 3-faces. \label{my6}
 \end{prop}
\begin{proof}
Let $v_1$ be a 5-vertex and let its neighbours be $v_2$, $v_3$, $v_4$, $v_5$ and $v_6$ in the cyclic order with $v_1v_2v_3$ being a 3-face. Let $v_1v_2v_7v_6$ be a 4-face and $v_2$ a 4-vertex with $v_8$ being the neighbour of $v_2$ different from $v_1$, $v_3$ and $v_7$. Let $f_1$ be the face partially bounded by $v_7v_2v_8$ and $f_2$ by $v_8v_2v_3$. Let $f_3$ be the second $\geq 4$-face adjacent to $v_1$. Let vertices $v_k$ and $v_{k+1}$ be neighbours of $v_1$ that are on face $f_3$. One of these configurations is depicted in Figure 7.
\begin{center}
\begin{tikzpicture}[scale = 0.5]
\node at (-3,2) {$f_1$};
\node at (0,3) {$f_2$};
\node at (2,0) {$f_3$};
\node (G-1) at (0,0) [circle,fill=black] [label=right:$v_1$] {};
\node (G-2) at (-1,2) [circle,fill=black] [label=left:$v_2$] {};
\node (G-3) at (1,2) [circle,fill=black] [label=right:$v_3$] {};
\node (G-4) at (1,-2) [circle,fill=black] [label=right:$v_4$] {};
\node (G-5) at (0,-2) [circle,fill=black] [label=below:$v_5$] {};
\node (G-6) at (-1,-2) [circle,fill=black] [label=left:$v_6$] {};
\node (G-7) at (-3,0) [circle,fill=black] [label=left:$v_7$] {};
\node (G-8) at (-2,3) [circle,fill=black] [label=left:$v_8$] {};
\draw[-][black] (G-1) edge (G-2);
\draw[-][black] (G-1) edge (G-3);
\draw[-][black] (G-1) edge (G-4);
\draw[-][black] (G-1) edge (G-5);
\draw[-][black] (G-1) edge (G-6);
\draw[-][black] (G-2) edge (G-3);
\draw[-][black] (G-2) edge (G-7);
\draw[-][black] (G-2) edge (G-8);
\draw[-][black] (G-4) edge (G-5);
\draw[-][black] (G-5) edge (G-6);
\draw[-][black] (G-6) edge (G-7);
\end{tikzpicture}
\hspace{0.5 in}
\begin{tikzpicture}[scale = 0.5]
\node at (-3,2) {$f_1'$};
\node at (0,3) {$f_2'$};
\node at (2,0) {$f_3'$};
\node (G-3) at (1,2) [circle,fill=black] [label=right:$v_3$] {};
\node (G-4) at (1,-2) [circle,fill=black] [label=right:$v_4$] {};
\node (G-5) at (0,-2) [circle,fill=black] [label=below:$v_5$] {};
\node (G-6) at (-1,-2) [circle,fill=black] [label=left:$v_6$] {};
\node (G-7) at (-3,0) [circle,fill=black] [label=left:$v_7$] {};
\node (G-8) at (-2,3) [circle,fill=black] [label=left:$v_8$] {};
\draw[-][black] (G-7) edge (G-8);
\draw[-][black] (G-8) edge (G-3);
\draw[-][black] (G-3) edge (G-4);
\draw[-][black] (G-4) edge (G-5);
\draw[-][black] (G-5) edge (G-6);
\draw[-][black] (G-6) edge (G-7);
\end{tikzpicture}\\
Figure 7: Configuration of Proposition \ref{my6} and reduction.
\end{center}

Let $H$ be the graph $G \backslash \{v_1, v_2\}$ with edges $v_6v_8$, $v_3v_8$ and $v_kv_{k+1}$ added. Let $f_1'$ be the new face in $H$ partially bounded by $v_7v_8$, $f_2'$ by $v_3v_8$ and $f_3'$ by $v_kv_{k+1}$. By the extremity of $G$ there is a $L$-colouring $c$ of $H$. As $v_6$ and $v_7$ are adjacent, they have different colours examine two cases: 

\paragraph{Case $c(v_6) > c(v_7)$.} Colour $v_2$ by a colour from $L(v_2)$ different from $c(v_3)$, $c(v_6)$, $c(v_7)$, $c(v_8)$, $c(f_1')$ and $c(f_2')$ call it $c(v_2)$. Colour $v_1$ a colour from $L(v_1)$ different from $c(v_2)$, $c(v_3)$, $c(v_4)$, $c(v_5)$, $c(v_6)$ and $c(f_3')$. This is a capital $L$-colouring as on the 4-face $v_1v_2v_7v_8$ the maximal colour is that of either $v_1$, $v_2$ or $v_6$.

\paragraph{Case $c(v_7) > c(v_6)$.} Colour $v_1$ by a colour from $L(v_1)$ different from $c(v_3)$, $c(v_4)$, $c(v_5)$, $c(v_6)$, $c(v_7)$ and $c(f_3')$ call it $c(v_1)$. Colour $v_2$ a colour from $L(v_2)$ different from $c(v_1)$, $c(v_3)$, $c(v_7)$, $c(v_8)$, $c(f_1')$ and $c(f_2')$. This is a capital $L$-colouring as on the 4-face $v_1v_2v_7v_8$ the maximal colour is that of either $v_1$, $v_2$ or $v_7$.\\

Therefore $G$ has an $L$-colouring.
\end{proof}

\begin{prop}
 In a extremal counter example $G$, no 3-face and 4-face can share an edge joining a 4-vertex and a 6-vertex incident with five 3-faces. \label{my7}
\end{prop}
\begin{proof}
Let $v_1$ be a 6-vertex and let its neighbours be $v_2$, $v_3$, $v_4$, $v_5$, $v_6$ and $v_7$ in the cyclic order. Let $v_1v_2v_3$, $v_1v_3v_4$, $v_1v_4v_5$, $v_1v_5v_6$, $v_1v_6v_7$ be 3-faces. Let $v_1v_2v_8v_7$ be a 4-face and $v_2$ is a 4-vertex. The remaining neighbour of $v_2$ is $v_9$. Let $f_1$ be the face partially bounded by $v_8v_2v_9$ and $f_2$ by $v_3v_2v_9$. This is demonstrated in figure 8.
\begin{center}
\begin{tikzpicture}[scale = 0.5]
\node at (-3,2) {$f_1$};
\node at (0,3) {$f_2$};
\node (G-1) at (0,0) [circle,fill=black] [label=left:$v_1$] {};
\node (G-2) at (-1,2) [circle,fill=black] [label=left:$v_2$] {};
\node (G-3) at (1,2) [circle,fill=black] [label=right:$v_3$] {};
\node (G-4) at (2,0) [circle,fill=black] [label=right:$v_4$] {};
\node (G-5) at (1,-2) [circle,fill=black] [label=right:$v_5$] {};
\node (G-6) at (0,-2) [circle,fill=black] [label=below:$v_6$] {};
\node (G-7) at (-1,-2) [circle,fill=black] [label=left:$v_7$] {};
\node (G-8) at (-3,0) [circle,fill=black] [label=left:$v_8$] {};
\node (G-9) at (-2,3) [circle,fill=black] [label=left:$v_9$] {};
\draw[-][black] (G-1) edge (G-2);
\draw[-][black] (G-1) edge (G-3);
\draw[-][black] (G-1) edge (G-4);
\draw[-][black] (G-1) edge (G-5);
\draw[-][black] (G-1) edge (G-6);
\draw[-][black] (G-1) edge (G-7);
\draw[-][black] (G-2) edge (G-3);
\draw[-][black] (G-2) edge (G-8);
\draw[-][black] (G-2) edge (G-9);
\draw[-][black] (G-3) edge (G-4);
\draw[-][black] (G-4) edge (G-5);
\draw[-][black] (G-5) edge (G-6);
\draw[-][black] (G-6) edge (G-7);
\draw[-][black] (G-7) edge (G-8);
\end{tikzpicture}
\hspace{0.5 in}
\begin{tikzpicture}[scale = 0.5]
\node at (-3,2) {$f_1'$};
\node at (0,3) {$f_2'$};
\node (G-3) at (1,2) [circle,fill=black] [label=right:$v_3$] {};
\node (G-4) at (2,0) [circle,fill=black] [label=right:$v_4$] {};
\node (G-5) at (1,-2) [circle,fill=black] [label=right:$v_5$] {};
\node (G-6) at (0,-2) [circle,fill=black] [label=below:$v_6$] {};
\node (G-7) at (-1,-2) [circle,fill=black] [label=left:$v_7$] {};
\node (G-8) at (-3,0) [circle,fill=black] [label=left:$v_8$] {};
\node (G-9) at (-2,3) [circle,fill=black] [label=left:$v_9$] {};
\draw[-][black] (G-8) edge (G-9);
\draw[-][black] (G-9) edge (G-3);
\draw[-][black] (G-3) edge (G-4);
\draw[-][black] (G-4) edge (G-5);
\draw[-][black] (G-5) edge (G-6);
\draw[-][black] (G-6) edge (G-7);
\draw[-][black] (G-7) edge (G-8);
\end{tikzpicture}\\
Figure 8: Configuration of Proposition \ref{my7} and reduction.
\end{center}

Let $H$ be the induced graph on $G \backslash \{v_1,v_2\}$ but with the added edges $v_8v_9$ and $v_3v_9$. Then by the extremity of $G$, $H$ has an $L$-colouring, $c$. Then as $v_7$ and $v_8$ are adjacent, they have different colours. We consider two cases:

\paragraph{Case $c(v_8) > c(v_7)$.} Colour $v_1$ by a colour from $L(v_1)$ different from $c(v_3)$, $c(v_4)$, $c(v_5)$, $c(v_6)$, $c(v_7)$ and $c(v_8)$, call it $c(v_1)$. Colour $v_2$ a colour from $L(v_2)$ different from $c(v_1)$, $c(v_3)$, $c(v_8)$, $c(v_9)$, $c(f_1')$ and $c(f_2')$. This is a capital $L$-colouring of $G$ as in the 4-face $v_1v_2v_8v_7$ the maximal colour is that of either $v_1$, $v_2$ or $v_8$.

\paragraph{Case $c(v_7) > c(v_8)$.} Colour $v_2$ by a colour from $L(v_2)$ different from $c(v_3)$, $c(v_7)$, $c(v_8)$, $c(v_9)$, $c(f_1')$ and $c(f_2')$ and call it $c(v_2)$. Then colour $v_1$ a colour from $L(v_1)$ different from $c(v_2)$, $c(v_3)$, $c(v_4)$, $c(v_5)$, $c(v_6)$ and $c(v_7)$. This is a capital $L$-colouring of $G$ as in the 4-face $v_1v_2v_8v_7$ the maximal colour is that of either $v_1$, $v_2$ or $v_7$.\\

So there is a $L$-colouring of $G$.
\end{proof}

\subsection{Discharging}

The existence of an extremal counter-example will be disproved by the discharging method. The initial charge of each $d(f)$-face $f \in F(G)$ is $d(f)-4$, and the initial charge of each $d(v)$-vertex $v \in V(G)$ is $d(v)-4$. By Euler's formula the total amount of charge is
\[
 \sum_{v \in V(G)} d(v) - 4 + \sum_{f \in F(G)} d(f) - 4 = 2 \vert E(G) \vert - 4 \vert V(G) \vert + 2 \vert E(G) \vert - 4 \vert F(G) \vert = -8.
\]
The initial charge is redistributed by the following rules.
\begin{itemize}
 \item[Rule V5]
 A 5-vertex incident with at most two 3-faces shall give each 3-face 1/2 units of charge. A 5-vertex incident with three 3-faces shall give each 3-face 1/3 units of charge.
 \item[Rule V6]
 A 6-vertex incident with at most four 3-faces shall give each 3-face 1/2 units of charge. A 6-vertex incident with five 3-faces shall give each 3-face 1/3 units of charge.
 \item[Rule V7]
 A 7-vertex incident with at most six 3-faces shall give each 3-face 1/2 units of charge. A 7-vertex incident with seven 3-faces shall give each 3-face 1/3 units of charge.
 \item[Rule V8]
 A $\geq 8$-vertex will give every 3-face 1/2 units of charge.
 \item[Rule E1]
 A $\geq 5$-face will give 1/2 units of charge to every 3-face adjacent via an edge joining two 4-vertices.
 \item[Rule E2]
 A $\geq 5$-face will give 1/6 units of charge to every 3-face adjacent via an edge joining a 4-vertex and a $\geq 5$-vertex.
\end{itemize}

Rules V5-8 do not allow a vertex to give out more charge than it started with, therefore any vertex after the application of the rules has non-negative charge. Also note that 4-faces are unaffected by the rules so they keep zero charge. 

\begin{lem}
Every $\geq 5$-face after the rules have been applied has non-negative charge. \label{my8}
\end{lem}
\begin{proof}
Consider a $\geq 5$-face, let $v_1$, $v_2$ and $v_3$ be 3 consecutive vertices on its boundary. If Rule E1 applies to the edge $v_1v_2$, then the other face partially bounded by $v_2v_3$ is a $\geq 4$-face by Lemma \ref{my4} Part 5. If the edge $v_1v_2$ uses Rule $E2$ with $v_2$ being the 4-vertex, then the other face containing $v_2v_3$ is a $\geq 4$-face. If the edge $v_1v_2$ uses Rule $E2$ with $v_1$ being the 4-vertex, then Rule E2 can apply with respect to $v_2v_3$. Therefore the $\geq 5$-face sends out through any two consecutive edges at most 1/2 units of charge. Therefore, $\geq 6$-faces have non-negative charge after discharging.\\

Let $f$ be a 5-face $v_1v_2v_3v_4v_5$. Note that the initial charge of $f$ is 1. If no edge on the boundary of $f$ uses Rule E1, then it gives out at most 5/6 units of charge. Suppose the $v_1v_2$ uses Rule E1, then $v_1$ and $v_2$ are 4-vertices. By Lemma \ref{my4} Part 5 the other faces containing $v_2v_3$ and $v_1v_5$ are $\geq 4$-faces, so no rule applies with respect to them. Since at most 1/2 units of charge is sent through $v_3v_4$ and $v_4v_5$ by the argument in the above paragraph, 5-faces have non-negative charge after the rules are applied.
\end{proof} 

It remains to consider 3-faces.

\begin{lem}
 Every 3-face after discharging has non-negitive charge. \label{my9}
\end{lem}
\begin{proof}
We distinguish cases based on how many 4-vertices are incident to a 3-face $f$. Recall that the initial charge of $f$ is -1.
 
\paragraph{Three 4-vertices} By Proposition \ref{my5} each edge of the 3-face is incident to a $\geq 5$-face. Therefore by Rule E1 the face $f$ receives 1/2 unit of charge from each of these three $\geq 5$-faces. So its final charge after discharging is 1/2.

\paragraph{Two 4-vertices} From Proposition \ref{my5} the edge of the 3-face that is joining the two 4-vertices is also contained in $\geq 5$-face. By Lemma \ref{my4} Part 5, the other two faces incident with $f$ are $\geq 4$-faces. Then there are three possibilities regarding the remaining vertex which we call $v$: $v$ is either $\geq 6$-vertex, a $5$-vertex incident with two or less 3-faces, or a 5-vertex incident with three 3-faces. These cases are represented in figure 9.

\begin{center}
\begin{tikzpicture}[scale = 1]
\node at (-5,-1) [circle,fill=black] [label=right:$\geq 5$-vertex] {};
\node at (-5,-2) [diamond,fill=black] [label=right:$4$-vertex] {};
\node (L-1) at (0,-2.5) {$\geq 5$-face};
\node (L-2) at (-1.5,-1) {$\geq 4$-face};
\node (L-3) at (1.5,-1) {$\geq 4$-face};
\node (G-1) at (0,0) [circle,fill=black] [label=right:$\geq 6$-vertex] {};
\node (G-2) at (-1,-2) [diamond,fill=black] {};
\node (G-3) at (1,-2) [diamond,fill=black] {};
\node (G-4) at (1,1) {};
\node (G-5) at (-1,1) {};
\node (G-6) at (-2,-2) {};
\node (G-7) at (-1,-3) {};
\node (G-8) at (2,-2) {};
\node (G-9) at (1,-3) {};
\node (G-10) at (0,-1) {};
\node (G-11) at (0,-1.5) {};
\draw[->][black] (G-1) edge (G-10);
\draw[->][black] (L-1) edge (G-11);
\draw[-][black] (G-1) edge (G-2);
\draw[-][black] (G-1) edge (G-3);
\draw[-][black] (G-1) edge (G-4);
\draw[-][black] (G-1) edge (G-5);
\draw[-][black] (G-2) edge (G-3);
\draw[-][black] (G-2) edge (G-6);
\draw[-][black] (G-2) edge (G-7);
\draw[-][black] (G-3) edge (G-8);
\draw[-][black] (G-3) edge (G-9);
\end{tikzpicture}
\hspace{0.5 in}
\begin{tikzpicture}[scale = 1]
\node at (0,-2.5) {$\geq 5$-face};
\node at (-1.5,-1) {$\geq 5$-face};
\node at (1.5,-1) {$\geq 5$-face};
\node (G-1) at (0,0) [circle,fill=black] [label=right:$5$-vertex with three 3-faces] {};
\node (G-2) at (-1,-2) [diamond,fill=black] {};
\node (G-3) at (1,-2) [diamond,fill=black] {};
\node (G-4) at (1,1) {};
\node (G-5) at (-1,1) {};
\node (G-6) at (-2,-2) {};
\node (G-7) at (-1,-3) {};
\node (G-8) at (2,-2) {};
\node (G-9) at (1,-3) {};
\node (G-10) at (0,-0.8) {};
\node (G-11) at (0,-1.5) {};
\node (G-12) at (-0.2,-1.3) {};
\node (G-13) at (0.2,-1.3) {};
\draw[->][black] (G-1) edge (G-10);
\draw[->][black] (L-1) edge (G-11);
\draw[->][black] (L-2) edge (G-12);
\draw[->][black] (L-3) edge (G-13);
\draw[-][black] (G-1) edge (G-2);
\draw[-][black] (G-1) edge (G-3);
\draw[-][black] (G-1) edge (G-4);
\draw[-][black] (G-1) edge (G-5);
\draw[-][black] (G-2) edge (G-3);
\draw[-][black] (G-2) edge (G-6);
\draw[-][black] (G-2) edge (G-7);
\draw[-][black] (G-3) edge (G-8);
\draw[-][black] (G-3) edge (G-9);
\end{tikzpicture}\\
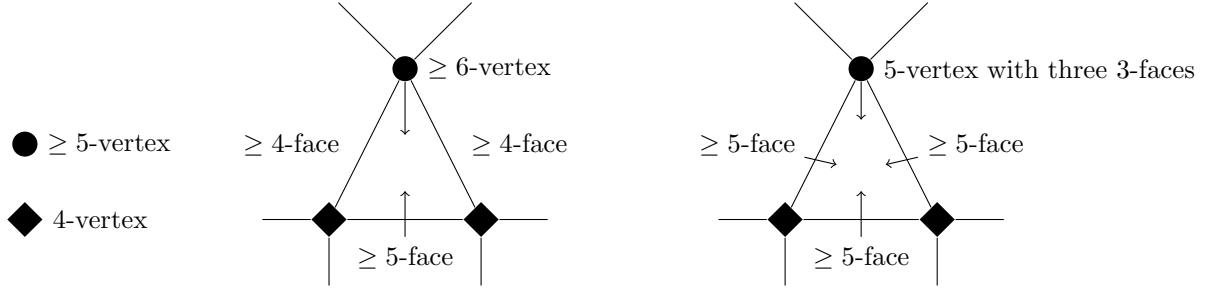
Figure 9: Two 2-vertex cases.
\end{center}

\subparagraph{The vertex $v$ is a $\geq 6$-vertex.} Note that the number of 3-faces incident with $v$ is at most $d(v) - 2$. So by Rules V6-8, the face receives 1/2 unit of charge from $v$. It also receives 1/2 unit of charge from the $\geq 5$-face containing the other two vertices of $f$ by Rule E1. Therefore the face $f$ has non-negative charge after discharging.

\subparagraph{The vertex $v$ is a $5$-vertex incident with two or less 3-faces.} By Rule V5 the 3-face receives 1/2 unit of charge from $v$. The 3-face also receives 1/2 unit of charge from the $\geq 5$-face containing the other two vertices from Rule E1. Therefore after discharging the face $f$ has non-negative charge.

\subparagraph{The vertex $v$ is a 5-vertex incident with three 3-faces.} By Proposition \ref{my6} the faces incident with $v$ that share an edge with the face $f$ are $\geq 5$-faces. By Rule V5 the face $f$ receives 1/3 units of charge from $v$. The face $f$ receives 1/6 unit of charge from each of the two $\geq 5$-faces sharing an edge containing $v$ by Rule E2. Finally, the face $f$ receives 1/2 unit of charge from the $\geq 5$-face sharing the edge not containing $v$. Therefore after discharging the faces has $1/6$ units of charge.

\paragraph{A single 4-vertex.} Let $v$ be one of the two $\geq 5$-vertices. The edge between $v$ and the 4-vertex is contained in the 3-face $f$ and a $\geq 4$-face by Lemma \ref{my4} Part 5. Therefore the number of 3-faces incident to $v$ is at most $d(v) - 1$. Then one of five cases happens with respect to $v$: $v$ is $\geq 7$-vertex, $v$ is a 6-vertex incident with four or less 3-faces, $v$ is a 5-vertex incident with two or less 3-faces, $v$ is a 6-vertex incident with five 3-faces, or $v$ is a 5-vertex incident with three 3-faces. See figure 10 for clarification.

\begin{center}
\begin{tikzpicture}[scale = 1]
\node at (-3,-1) [circle,fill=black] [label=right:$\geq 5$-vertex] {};
\node at (-3,-2) [diamond,fill=black] [label=right:$4$-vertex] {};
\node (L-1) at (1.5,-1) {$\geq 4$-face};
\node (G-1) at (0,0) [diamond,fill=black] {};
\node (G-2) at (1,-2) [circle,fill=black] [label=below:$\geq 7$-vertex] {};
\node (G-3) at (-0.1,-2) {};
\node (G-4) at (1,1) {};
\node (G-5) at (2,-2) {};
\node (G-6) at (0,1) {};
\node (G-7) at (0,-3) {};
\node (G-8) at (0,-1.2) {};
\draw[-][black] (G-1) edge (G-2);
\draw[-][black] (G-1) edge (G-4);
\draw[-][black] (G-2) edge (G-3);
\draw[-][black] (G-2) edge (G-5);
\draw[->][black] (G-2) edge (G-8);
\draw[dotted][black] (G-6) edge (G-7);
\end{tikzpicture}
\hspace{0.3 in}
\begin{tikzpicture}[scale = 1]
\node (L-1) at (1.5,-1) {$\geq 5$-face};
\node (G-1) at (0,0) [diamond,fill=black] {};
\node (G-2) at (1,-2) [circle,fill=black] [label=below:$6$-vertex with five 3-faces] {};
\node (G-3) at (-0.1,-2) {};
\node (G-4) at (1,1) {};
\node (G-5) at (2,-2) {};
\node (G-6) at (0,1) {};
\node (G-7) at (0,-3) {};
\node (G-8) at (0.3,-1.5) {};
\node (G-9) at (0.2,-1.3) {};
\draw[-][black] (G-1) edge (G-2);
\draw[-][black] (G-1) edge (G-4);
\draw[-][black] (G-2) edge (G-3);
\draw[-][black] (G-2) edge (G-5);
\draw[->][black] (G-2) edge (G-8);
\draw[->][black] (L-1) edge (G-9);
\draw[dotted][black] (G-6) edge (G-7);
\end{tikzpicture}
\hspace{0.3 in}
\begin{tikzpicture}[scale = 1]
\node (L-1) at (1.5,-1) {$\geq 5$-face};
\node (G-1) at (0,0) [diamond,fill=black] {};
\node (G-2) at (1,-2) [circle,fill=black] [label=below:$5$-vertex with three 3-faces] {};
\node (G-3) at (-0.1,-2) {};
\node (G-4) at (1,1) {};
\node (G-5) at (2,-2) {};
\node (G-6) at (0,1) {};
\node (G-7) at (0,-3) {};
\node (G-8) at (0.3,-1.5) {};
\node (G-9) at (0.2,-1.3) {};
\draw[-][black] (G-1) edge (G-2);
\draw[-][black] (G-1) edge (G-4);
\draw[-][black] (G-2) edge (G-3);
\draw[-][black] (G-2) edge (G-5);
\draw[->][black] (G-2) edge (G-8);
\draw[->][black] (L-1) edge (G-9);
\draw[dotted][black] (G-6) edge (G-7);
\end{tikzpicture}\\
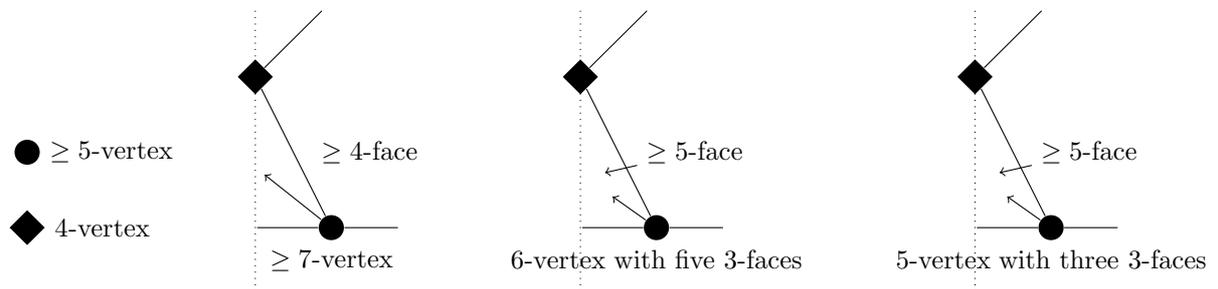
Figure 10: Three half cases.
\end{center}

\subparagraph{The vertex $v$ is $\geq 7$-vertex, $v$ is a 6-vertex incident with four or less 3-faces or $v$ is a 5-vertex with incident with two or less 3-faces.} In all these cases the face $f$ receives at least 1/2 units of charge from $v$ by Rule V5-8.

\subparagraph{The vertex $v$ is a 6-vertex incident with five 3-faces.} By Proposition \ref{my7} the edge between $v$ and the 4-vertex is contained in a $\geq 5$-face. By Rule E2, the 3-face receives 1/6 units of charge from the $\geq 5$-face. By Rule V6, the 3-face also receives 1/3 units of charge from $v$.

\subparagraph{The vertex $v$ is a 5-vertex incident with three 3-faces.} By Proposition \ref{my6} the edge between $v$ and the 4-vertex is adjacent to a $\geq 5$-face.  By Rule E2, the 3-face receives 1/6 units of charge from the $\geq 5$-face. By Rule V6, the 3-face receives 1/3 units of charge from $v$.\\

In each of the cases the face receives at least 1/2 unit of charge together from the vertex $v$ and the edge containing $v$ and the 4-vertex. So, the 3-face $f$ has non-negative charge after discharging.

\paragraph{No 4-vertex} By Rules V5-8 the face receives at least 1/3 unit of charge from each vertex it contains. Therefore the 3-face has non-negative charge after discharging.
\end{proof}

\begin{proof}[Proof of Theorem \ref{my3}]
 From discussion after discharging rules, Lemma \ref{my8}, and Lemma \ref{my9}, we have that every vertex and face after discharging has non-negative charge. This contradicts the amount of charge in the system. So, there is no extremal counter-example to Theorem \ref{my3}.
\end{proof}

\section{Conclusion}

In the case of ordinary colouring, it is known that every plane graph is 5-choosable \cite{CT94}. The bound is tight as shown by Voigt \cite{MV93}. We have tried to construct an example of a plane graph with lists of sizes five such that there is no capital colouring assigning each vertex a colour from its list. However, we have not managed to have done so, which led us to propose the following.
\begin{conj}
If each vertex of a plane graph is assigned a list of five integers,
then there exists a capital colouring assigning each vertex a colour from its list.
\end{conj}

\section{Acknowledgements}

\tab I would like to thank Dan Kr\'{a}l' for supervising me over summer in the Undergraduate Research Support Scheme at Warwick University and the university for funding such projects. I would also like to thank Zden\v{e}k Dvo\v{r}\'{a}k and Ond\v{r}ej Pangr\'{a}c at Charles University for helping me whilst I visited their university as well as the university itself for the support it offered for my trip. Thanks to Tom\'{a}\v{s} Kaiser and the University of West Bohemia for help and support on my trip. I would like to thank Igor Fabrici for his comments on the presentation and the referees for there useful comments and suggestions.


\begin{thebibliography}{9}

\bibitem{KAHW}
Appel, K. and Haken, W. \emph{The Solution of the Four-Color Map Problem.} Sci. Amer. 237, 108-121, 1977.

\bibitem{IFFG}
Fabrici, I. G\"{o}ring, F. \emph{Unique-maximum colouring of plane graphs}. Manuscript, 2013.

\bibitem{HG59}
Gr\"{o}tzsch, H. \emph{Zur Theorie der diskreten Gebilde, VII: Ein Dreifarbensatz f\"{u}r  dreikreisfreie Netze auf der Kugel}. Wiss. Z. Martin-Luther-U., Halle-Wittenberg, Math.-Nat. Reihe 8: 109–120.

\bibitem{RSST}
Robertson, N. Sanders, D. P. Seymour, P. D. and Thomas, R. \emph{A New Proof of the Four Colour Theorem.} Electron. Res. Announc. Amer. Math. Soc. 2, 17-25, 1996.

\bibitem{CT94}
Thomassen, C. \emph{Every Planar Graph is 5-Choosable.} Journal of Cominatorial theory, Series B 62, 180-181, 1994.

\bibitem{MV93}
Voigt, M. \emph{List colourings of planar graphs.} Discrete Mathematics Volume 120, Issues 1–3, 215–219, 1993.


\end{thebibliography}
\end{document}